\documentclass{article}
\usepackage[english]{babel}
\usepackage[utf8x]{inputenc}
\usepackage[T1]{fontenc}
\usepackage[color,matrix,arrow]{xy}
\usepackage{tikz}
\usetikzlibrary{decorations.markings,arrows,automata,arrows,backgrounds,snakes}

\usepackage{listings}
\usepackage{color}
\usepackage{algpseudocode}
\usepackage{algorithm}
\usepackage{comment}
\definecolor{dkgreen}{rgb}{0,0.6,0}
\definecolor{gray}{rgb}{0.5,0.5,0.5}
\definecolor{mauve}{rgb}{0.58,0,0.82}

\usepackage{amsmath}
\usepackage{amssymb}
\usepackage{asymptote}
\usepackage[nottoc]{tocbibind}
\usepackage{amsfonts}
\usepackage{graphicx}
\usepackage[colorinlistoftodos]{todonotes}
\usepackage[colorlinks=true, allcolors=blue]{hyperref}
\usepackage{float}
\usepackage{mathtools}
\usepackage{amssymb,amsmath}
\usepackage{tikz}
\usepackage{amsthm}
\theoremstyle{definition}

\newtheorem{lemma}{Lemma}[section]
\newtheorem{definition}{Definition}[section]

\newcommand{\ep}    {\epsilon}

\newcommand{\RR}    {\mathbb{R}}
\newtheorem{thm}{Theorem}%[section]

\newtheorem{cor}[thm]   {Corollary}
\newtheorem{rem}[thm]   {Remark}
\newtheorem{defn}[thm]  {Definition}
\newtheorem{ex}[thm]    {Example}
\newtheorem{prop}[thm]  {Proposition}

\newcounter{foo}  \Alph{foo}

\newenvironment{example}
{\medskip\par\noindent{\sc Example}\ }
{\par}

\title{Merge trees in discrete Morse theory}
\author{Benjamin Johnson, Nicholas A. Scoville}
\date{July 2020}

\begin{document}
\tikzset{->-/.style={decoration={
  markings,
  mark=at position .5 with {{->}}},postaction={decorate}}}

\maketitle

\begin{abstract}
In this paper, we study merge trees induced by a discrete Morse function on a tree.  Given a discrete Morse function, we provide a method to constructing an induced merge tree and define a new notion of equivalence of discrete Morse functions based on the induced merge tree.  We then relate the matching number of a tree to a certain invariant of the induced merge tree. Finally, we count the number of merge trees that can be induced on a star graph and characterize the induced merge tree.
\end{abstract}

\section{Introduction}

Topological data analysis seeks to understand a set of data by studying topological properties of that data.  One highly successful tool in this regard is persistent homology.  Persistence has been used to study statistical mechanics \cite{mech17}, hypothesis testing \cite{Hyp14}, image analysis \cite{Carl09}, complex networks \cite{Sco17}, and many other phenomena.  Recently, there has been interest in studying merge trees, a special kind of persistence \cite{Lamdge,Oster}. Part of the advantage of studying a merge tree instead of the persistence diagram is that the merge tree gives more detailed information about precisely which components merged with which other components.  It tracks not only the lifetime of a component but its evolution as well.

In \cite{Curry-2017}, Justin Curry studies functions on the unit interval that have the same persistent homology.  In this smooth setting, Curry develops a merge tree associated to a Morse set, an abstraction of path components associated to a Morse function on a compact, connected manifold. He is then able to count merge trees under a suitable notion of equivalence.  In this paper, we take up a similar problem in a purely discrete setting; that is, given a discrete Morse function on a tree (i.e. 1-dimensional abstract simplicial complex), we associate a tree, appropriately called a merge tree (Definition \ref{thm: construct mereg tree}). We describe a method to obtain a merge tree from a discrete Morse function on a tree in Theorem \ref{thm: construct mereg tree}.  After defining a notion of equivalence of merge trees, we prove that a certain invariant of an induced merge tree of $T$ yields a lower bound for the matching number of $T$ in Proposition \ref{prop: impasse matching}. Section \ref{Comparison with other notions of equivalence} is devoted to comparing merge equivalent discrete Morse functions with other notions of equivalence of discrete Morse functions.   Then in Section \ref{Merge tree of a star graph}, we give a characterization of merge trees induced by a discrete Morse function on a star graph. Finally, we share some future directions in Section \ref{Conclusion and future work}.

\section{Background}

\subsection{Graphs and trees}

Let $G=(V(G),E(G))$ be a finite, loopless graph without multi-edges (i.e. a $1$-dimensional abstract simplicial complex).   We call an edge or a vertex of $G$ a \textbf{simplex}. If $e=uv$ is an edge, we say that the edge $e$ is \textbf{incident} with vertex $v$ and that $u$ and $v$ are adjacent. We use $|V(G)|$ to denote the number of vertices of $G$ and $|E(G)|$ to denote the number of edges of $G$.\\

We work exclusively with trees in this paper.  Here we recall several important characterizations of trees.  They will be utilized without further reference.

\begin{thm}\label{thm char of trees}(Characterization of trees) Let $G$ be a connected graph with $v$ vertices and $e$ edges. The following are equivalent:
\begin{enumerate}
    \item[a)] Every two vertices of $G$ are connected by a unique path.
    \item[b)] $v=e+1.$
    \item[c)] $G$ contains no cycles.
    \item[d)] $b_1(G)=0$ where $b_1$ is the first Betti number of $G$ (\cite[Chapter II.4]{Renzo}).
    \item[e)] The removal of any edge from $G$ results in a disconnected graph.
\end{enumerate}

\end{thm}

\noindent A connected graph that satisfies any of the above characterizations is called a \textbf{tree}. Proofs of the equivalence of the statements may be found in any graph theory textbook (e.g. \cite[Chapter 2.2]{Chartrand16}). A disconnected graph $F$ such that each component of $F$ is a tree is called a \textbf{forest}.  For any vertex $v\in F$, we let $F[v]$ denote the connected component of $F$ containing $v$. It immediately follows that if $F$ is a forest with two distinct vertices $u,v\in F$, then there is a path between two vertices $u$ and $v$ if and only if $F[u]=F[v]$.

\subsection{Discrete Morse theory}

Our references for the basics of discrete Morse theory are \cite{Forman-2002,KnudsonBook, Koz20, DMTSco}.  There are several different ways of viewing a discrete Morse function.  For our purposes, we make the following definition:

\begin{defn}\label{discreteMF}  Let $G$ be a graph.  A function $f\colon G \to \RR$ is called  \textbf{weakly increasing} if $f(v)\leq f(e)$ whenever $v \subseteq e$.  A \textbf{discrete Morse function} $f\colon G \to \RR$ is a weakly increasing function which is at most 2--1 and satisfies the property that if $f(\sigma)=f(\tau)$, then either $\sigma\subseteq \tau$ or $\tau\subseteq \sigma$. Any simplex $\sigma$ on which $f$ is 1--1 is called \textbf{critical} and the value $f(\sigma)$ is a \textbf{critical value} of $f$.
\end{defn}

\begin{ex}\label{ex: first dmf}  Define the function $f$ on $T$ as follows:
$$
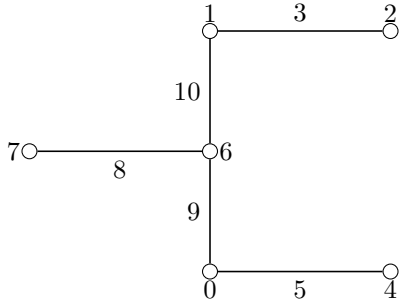
\begin{tikzpicture}[scale=.8]

\node[inner sep=2pt, circle] (0) at (0,0) [draw] {};
\node[inner sep=2pt, circle] (1) at (3,0) [draw] {};
\node[inner sep=2pt, circle] (2) at (-3,2) [draw] {};
\node[inner sep=2pt, circle] (3) at (0,2) [draw] {};
\node[inner sep=2pt, circle] (4) at (0,4) [draw] {};
\node[inner sep=2pt, circle] (5) at (3,4) [draw] {};

\path[style=semithick] (0) edge node[anchor=north]{{$5$}}(1);
\path[style=semithick] (0) edge node[anchor=east]{{$9$}}(3);
\path[style=semithick] (2) edge node[anchor=north]{{$8$}}(3);
\path[style=semithick] (3) edge node[anchor=east]{{$10$}}(4);
\path[style=semithick] (4) edge node[anchor=south]{{$3$}}(5);

\node[anchor = north]  at (0) {{$0$}};
\node[anchor = north]  at (1) {{$4$}};
\node[anchor = east]  at (2) {{$7$}};
\node[anchor = west]  at (3) {{$6$}};
\node[anchor = south]  at (4) {{$1$}};
\node[anchor = south]  at (5) {{$2$}};
\end{tikzpicture}
$$

\noindent Then $f$ is a discrete Morse function.  Note that all values are critical.
\end{ex}

\begin{defn}\label{level} Let $G$ be a graph.  Given $a \in \mathbb R$ the \textbf{level subcomplex} $G_a$ is defined to be the induced subgraph of $G$ consisting of all simplices $\sigma$ with $f(\sigma) \leq a$.  For each critical value $c_0< \ldots < c_{m-1}$ of $f$, we consider the induced sequence of level subcomplexes $\{v\}=G_{c_0} \subset G_{c_1}\subset \ldots \subset G_{c_{m-1}}$.  In the sequel, we will use the notation $G_{c_i-\ep}$ to denote the level subcomplex immediately preceding $G_{c_i}$; that is, $\ep$ is chosen so that $f(\sigma)<c_i-\ep<c_i$ for every $\sigma\in G$ such that $f(\sigma)<c_i$.
\end{defn}

\section{Merge trees}

In this section we introduce merge trees, our main object of study.

\subsection{Basics of merge trees}

\begin{defn}\label{defn merge tree} A \textbf{binary tree} is a rooted tree where each vertex has at most two children, and each child is designated as its \textbf{left} (L) or \textbf{right} (R) child.  A binary tree is \textbf{full} if every vertex has $0$ or $2$ children. A \textbf{merge tree} is a full binary tree. A node with exactly one neighbor is a \textbf{leaf node} or \textbf{leaf}. Otherwise, a node with more than one neighbor is an \textbf{internal node}.
\end{defn}

Although graphs and merge trees are different objects, they both look the same and consist of vertices and edges.  To help distinguish them, we reserve the term ``node" for merge trees and ``vertex" for graphs.

\begin{rem}\label{rem: parent/child} We will view a merge tree upside down from how one normally views a binary tree; that is, the root is drawn at the bottom of the tree, as opposed to the top. However, we maintain the parent/child relationship language so that when considering the merge tree, a child will be above the parent.
\end{rem}

\begin{ex}
Consider the merge tree below:

$$
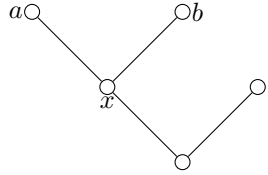
\begin{tikzpicture}

\node[inner sep=2pt, circle] (0) at (0,2) [draw] {};
\node[inner sep=2pt, circle] (1) at (2,2) [draw] {};
\node[inner sep=2pt, circle] (2) at (3,1) [draw] {};
\node[inner sep=2pt, circle] (3) at (1,1) [draw] {};
\node[inner sep=2pt, circle] (4) at (2,0) [draw] {};

\draw[-]  (0)--(3) node[midway, below] {};
\draw[-]  (3)--(1) node[midway, below] {};
\draw[-]  (3)--(4) node[midway, below] {};
\draw[-]  (2)--(4) node[midway, below] {};

\node[anchor = east]  at (0) {{$a$}};
\node[anchor = west]  at (1) {{$b$}};
\node[anchor = north]  at (3) {{$x$}};

\end{tikzpicture}
$$

To illustrate Remark \ref{rem: parent/child}, we say that $a$ and $b$ are children of $x$, even though $x$ is below $a$ and $b$.
\end{ex}

\begin{rem}
Note that the left and right information is part of the definition of a merge tree so that

$$
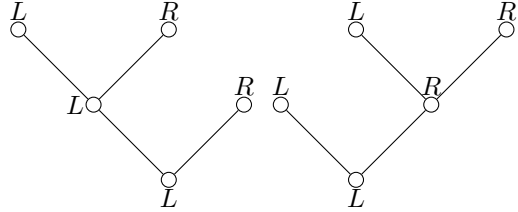
\begin{tikzpicture}

\node[inner sep=2pt, circle] (0) at (0,0) [draw] {};
\node[inner sep=2pt, circle] (1) at (-1,1) [draw] {};
\node[inner sep=2pt, circle] (2) at (1,1) [draw] {};
\node[inner sep=2pt, circle] (3) at (-2,2) [draw] {};
\node[inner sep=2pt, circle] (4) at (0,2) [draw] {};

\draw[-]  (0)--(1) node[midway, below] {};
\draw[-]  (0)--(2) node[midway, below] {};
\draw[-]  (1)--(4) node[midway, below] {};
\draw[-]  (1)--(3) node[midway, below] {};

\node[anchor = north]  at (0) {{$L$}};
\node[anchor = east]  at (1) {{$L$}};
\node[anchor = south]  at (2) {{$R$}};
\node[anchor = south]  at (3) {{$L$}};
\node[anchor = south]  at (4) {{$R$}};

\end{tikzpicture}
\begin{tikzpicture}

\node[inner sep=2pt, circle] (0) at (0,0) [draw] {};
\node[inner sep=2pt, circle] (1) at (-1,1) [draw] {};
\node[inner sep=2pt, circle] (2) at (1,1) [draw] {};
\node[inner sep=2pt, circle] (3) at (0,2) [draw] {};
\node[inner sep=2pt, circle] (4) at (2,2) [draw] {};

\draw[-]  (0)--(1) node[midway, below] {};
\draw[-]  (0)--(2) node[midway, below] {};
\draw[-]  (2)--(4) node[midway, below] {};
\draw[-]  (2)--(3) node[midway, below] {};

\node[anchor = north]  at (0) {{$L$}};
\node[anchor = south]  at (1) {{$L$}};
\node[anchor = south]  at (2) {{$R$}};
\node[anchor = south]  at (3) {{$L$}};
\node[anchor = south]  at (4) {{$R$}};

\end{tikzpicture}
$$
are different merge trees even though they are isomorphic as graphs. We will sometimes suppress the L and R labelings below, as the position on the page makes a node's L or R status clear.

\end{rem}

To any discrete Morse function on a tree, we are able to associate a merge tree through the following construction.

\begin{thm}\label{thm: construct mereg tree} Let $f\colon T \to \RR$ be a discrete Morse function on a tree $T$. Then $f$ induces a merge tree $M_f=M$.
\end{thm}

\begin{proof}
Let $f\colon T \to \RR$ be a discrete Morse function with critical values $c_0<c_1<c_2<\ldots < c_m$ and write $C:=\{c_0, c_1, \ldots, c_m\}$. We will construct a merge tree whose node set is in 1--1 correspondence with $C$. In order to organize information, we will label each node of the merge tree by defining a function $f_M$ that takes in nodes of $M$ and yields real numbers. Furthermore, each node will be given a direction L or R. We construct $M$ inductively on the critical edges of $f$ in reverse order.\\

Begin by creating a node $n_{c_m}$ corresponding to the critical edge in $T$ labeled $c_m$.  Define $f_M(n_{c_m}):= c_m$ along with direction L. \\

Inductively, let $n_{c_i}$ be a node of $M$ corresponding to a critical edge $uv\in C$.   Create two child nodes of $n_{c_i}$  called $n_v$ and $n_u$.  Define $f_M(n_u):=\max\{f(\sigma) \colon \sigma \in T_{c_{i}-\ep}[u], \sigma \text{ critical}\}$ and $f_M(n_v):=\max\{f(\sigma) \colon \sigma \in T_{c_{i}-\ep}[v], \sigma \text{ critical}\}$ (see Definition \ref{level} for meaning of $T_{c_i-\ep}$).  Note that the values of the child nodes can be values from a critical vertex or a critical edge.  If $\min\{f(\sigma) \colon \sigma \in T_{c_{i}-\ep}[u]\} < \min\{f(\sigma) \colon \sigma \in T_{c_{i}-\ep}[v], \sigma \text{ critical}\}$, then give $n_u$ the same direction (L or R) as that of $n_{c_i}$ and give $n_v$ the opposite direction.

Continue over all critical edges to obtain $M$.
\end{proof}

It can be difficult to build the induced merge tree starting from the ``top down" or the smallest value of the discrete Morse function since when adding new nodes to the merge tree, it is often unclear where a node is placed on the merge tree.  This is because where it is placed depends on which component(s) it ends up merging to and when.  Fortunately, Theorem \ref{thm: construct mereg tree} is  starting from the ``bottom up" or the largest value of the discrete Morse function.  We give an example below.

\begin{ex} To illustrate the construction of Theorem \ref{thm char of trees}, we will take the  discrete Morse function from Example \ref{ex: first dmf}.  We begin by identifying the critical edge values and placing them in reverse order:  $10, 9, 8, 5, 3$. The largest value is 10, so it corresponds to a node in $M$ with label 10 and direction L:

$$
\begin{tikzpicture}

\node[inner sep=2pt, circle] (10) at (5,0) [draw] {};

\node[anchor = south]  at (10) {{$10L$}};
\end{tikzpicture}
$$

We then look at the level subcomplex $T_{10-\ep}$ and identify the largest value in each of the trees that were incident with 10.
$$
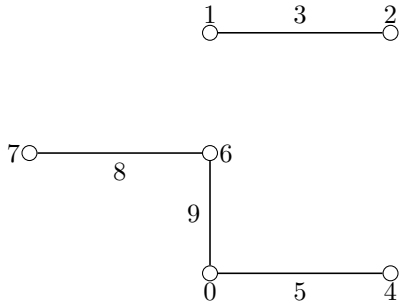
\begin{tikzpicture}[scale=.8]

\node[inner sep=2pt, circle] (0) at (0,0) [draw] {};
\node[inner sep=2pt, circle] (1) at (3,0) [draw] {};
\node[inner sep=2pt, circle] (2) at (-3,2) [draw] {};
\node[inner sep=2pt, circle] (3) at (0,2) [draw] {};
\node[inner sep=2pt, circle] (4) at (0,4) [draw] {};
\node[inner sep=2pt, circle] (5) at (3,4) [draw] {};

\path[style=semithick] (0) edge node[anchor=north]{{$5$}}(1);
\path[style=semithick] (0) edge node[anchor=east]{{$9$}}(3);
\path[style=semithick] (2) edge node[anchor=north]{{$8$}}(3);
\path[style=semithick] (4) edge node[anchor=south]{{$3$}}(5);

\node[anchor = north]  at (0) {{$0$}};
\node[anchor = north]  at (1) {{$4$}};
\node[anchor = east]  at (2) {{$7$}};
\node[anchor = west]  at (3) {{$6$}};
\node[anchor = south]  at (4) {{$1$}};
\node[anchor = south]  at (5) {{$2$}};

\end{tikzpicture}
$$

In this case, the two values are $3$ and $9$.  To determine which is to the left and which is to the right, we look for the tree with the minimum value.  In this case, $0<1$ so that $9$ shares the same direction as 10.  We thus obtain

$$
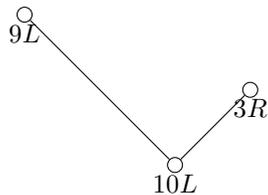
\begin{tikzpicture}

\node[inner sep=2pt, circle] (3) at (6,1) [draw] {};

\node[inner sep=2pt, circle] (9) at (3,2) [draw] {};
\node[inner sep=2pt, circle] (10) at (5,0) [draw] {};

\draw[-]  (3)--(10) node[midway, below] {};
\draw[-]  (10)--(9) node[midway, below] {};

\node[anchor = north]  at (3) {{$3R$}};
\node[anchor = north]  at (9) {{$9L$}};
\node[anchor = north]  at (10) {{$10L$}};
\end{tikzpicture}
$$

We move next to $9$, and consider the level subcomplex $T_{9-\ep}$:
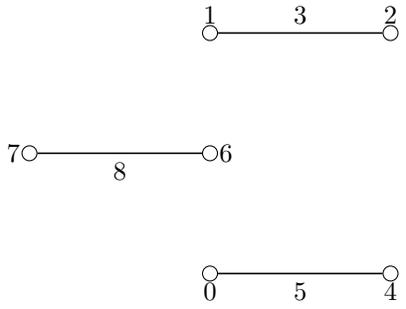
$$
\begin{tikzpicture}[scale=.8]

\node[inner sep=2pt, circle] (0) at (0,0) [draw] {};
\node[inner sep=2pt, circle] (1) at (3,0) [draw] {};
\node[inner sep=2pt, circle] (2) at (-3,2) [draw] {};
\node[inner sep=2pt, circle] (3) at (0,2) [draw] {};
\node[inner sep=2pt, circle] (4) at (0,4) [draw] {};
\node[inner sep=2pt, circle] (5) at (3,4) [draw] {};

\path[style=semithick] (0) edge node[anchor=north]{{$5$}}(1);
\path[style=semithick] (2) edge node[anchor=north]{{$8$}}(3);
\path[style=semithick] (4) edge node[anchor=south]{{$3$}}(5);

\node[anchor = north]  at (0) {{$0$}};
\node[anchor = north]  at (1) {{$4$}};
\node[anchor = east]  at (2) {{$7$}};
\node[anchor = west]  at (3) {{$6$}};
\node[anchor = south]  at (4) {{$1$}};
\node[anchor = south]  at (5) {{$2$}};

\end{tikzpicture}
$$

Now $9$ was connected to $6$ and $0$, and the maximum value on each of their trees is $8$ and $5$, respectively, so these will be the labels of the two new nodes above $9$.  To see which one shares the direction with $9$, we see that the tree with the vertex $0$ has minimum value, so $5$ shares the same direction as $9$.  We then obtain

$$
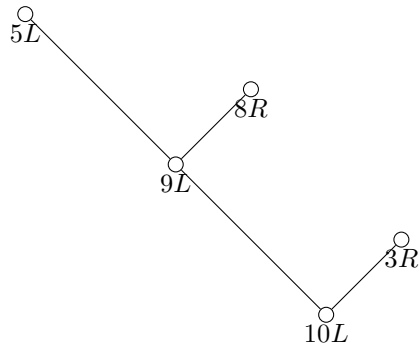
\begin{tikzpicture}

\node[inner sep=2pt, circle] (3) at (6,1) [draw] {};
\node[inner sep=2pt, circle] (5) at (1,4) [draw] {};
\node[inner sep=2pt, circle] (8) at (4,3) [draw] {};
\node[inner sep=2pt, circle] (9) at (3,2) [draw] {};
\node[inner sep=2pt, circle] (10) at (5,0) [draw] {};

\draw[-]  (5)--(9) node[midway, below] {};
\draw[-]  (8)--(9) node[midway, below] {};

\draw[-]  (3)--(10) node[midway, below] {};
\draw[-]  (10)--(9) node[midway, below] {};

\node[anchor = north]  at (3) {{$3R$}};

\node[anchor = north]  at (5) {{$5L$}};
\node[anchor = north]  at (8) {{$8R$}};
\node[anchor = north]  at (9) {{$9L$}};
\node[anchor = north]  at (10) {{$10L$}};
\end{tikzpicture}
$$

Now in $T_{8-\ep}$, $8$ was connected to the isolated vertex $7$ and isolated vertex $6$.  Hence the two new nodes connected to $8$ will be $6$ and $7$.  Since $6<7$, $6$ and $8$ share the same direction yielding

$$
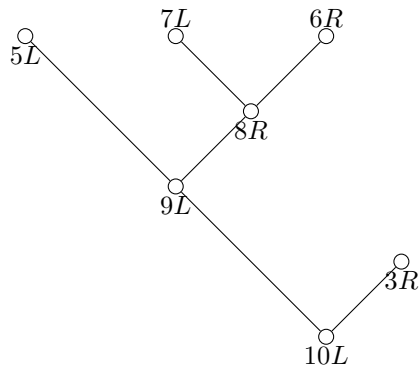
\begin{tikzpicture}

\node[inner sep=2pt, circle] (6) at (5,4) [draw] {};
\node[inner sep=2pt, circle] (3) at (6,1) [draw] {};
\node[inner sep=2pt, circle] (5) at (1,4) [draw] {};
\node[inner sep=2pt, circle] (8) at (4,3) [draw] {};
\node[inner sep=2pt, circle] (9) at (3,2) [draw] {};
\node[inner sep=2pt, circle] (10) at (5,0) [draw] {};
\node[inner sep=2pt, circle] (7) at (3,4) [draw] {};

\draw[-]  (5)--(9) node[midway, below] {};
\draw[-]  (8)--(9) node[midway, below] {};

\draw[-]  (3)--(10) node[midway, below] {};
\draw[-]  (10)--(9) node[midway, below] {};
\draw[-]  (7)--(8) node[midway, below] {};
\draw[-]  (6)--(8) node[midway, below] {};

\node[anchor = south]  at (6) {{$6R$}};
\node[anchor = south]  at (7) {{$7L$}};
\node[anchor = north]  at (3) {{$3R$}};

\node[anchor = north]  at (5) {{$5L$}};
\node[anchor = north]  at (8) {{$8R$}};
\node[anchor = north]  at (9) {{$9L$}};
\node[anchor = north]  at (10) {{$10L$}};
\end{tikzpicture}
$$

The next critical edge value is $5$, so we consider the level subcomplex $T_{5-\ep}$:
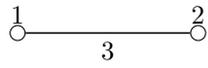
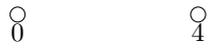
$$
\begin{tikzpicture}[scale=.8]

\node[inner sep=2pt, circle] (0) at (0,0) [draw] {};
\node[inner sep=2pt, circle] (1) at (3,0) [draw] {};
\node[inner sep=2pt, circle] (4) at (0,4) [draw] {};
\node[inner sep=2pt, circle] (5) at (3,4) [draw] {};

\path[style=semithick] (4) edge node[anchor=north]{{$3$}}(5);

\node[anchor = north]  at (0) {{$0$}};
\node[anchor = north]  at (1) {{$4$}};
\node[anchor = south]  at (4) {{$1$}};
\node[anchor = south]  at (5) {{$2$}};

\end{tikzpicture}
$$

The edge $5$ was connected to isolated vertices $0$ and $4$, yielding
$$
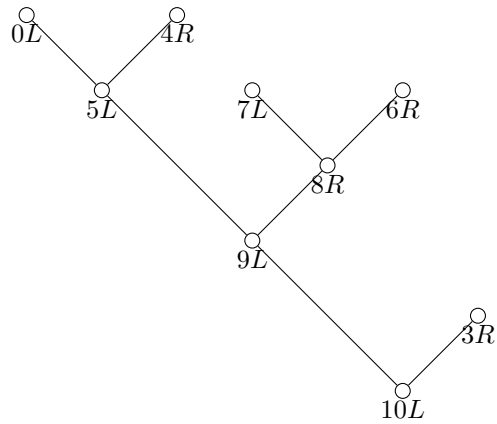
\begin{tikzpicture}

\node[inner sep=2pt, circle] (0) at (0,5) [draw] {};
\node[inner sep=2pt, circle] (3) at (6,1) [draw] {};
\node[inner sep=2pt, circle] (4) at (2,5) [draw] {};
\node[inner sep=2pt, circle] (5) at (1,4) [draw] {};
\node[inner sep=2pt, circle] (6) at (5,4) [draw] {};
\node[inner sep=2pt, circle] (7) at (3,4) [draw] {};
\node[inner sep=2pt, circle] (8) at (4,3) [draw] {};
\node[inner sep=2pt, circle] (9) at (3,2) [draw] {};
\node[inner sep=2pt, circle] (10) at (5,0) [draw] {};

\draw[-]  (0)--(5) node[midway, below] {};
\draw[-]  (4)--(5) node[midway, below] {};
\draw[-]  (5)--(9) node[midway, below] {};
\draw[-]  (7)--(8) node[midway, below] {};
\draw[-]  (6)--(8) node[midway, below] {};
\draw[-]  (8)--(9) node[midway, below] {};
\draw[-]  (3)--(10) node[midway, below] {};
\draw[-]  (10)--(9) node[midway, below] {};

\node[anchor = north]  at (0) {{$0L$}};
\node[anchor = north]  at (3) {{$3R$}};
\node[anchor = north]  at (4) {{$4R$}};
\node[anchor = north]  at (5) {{$5L$}};
\node[anchor = north]  at (6) {{$6R$}};
\node[anchor = north]  at (7) {{$7L$}};
\node[anchor = north]  at (8) {{$8R$}};
\node[anchor = north]  at (9) {{$9L$}};
\node[anchor = north]  at (10) {{$10L$}};
\end{tikzpicture}
$$

Finally, $3$ is connected to $2$ and $1$, giving us the merge tree induced by the discrete Morse function:

$$
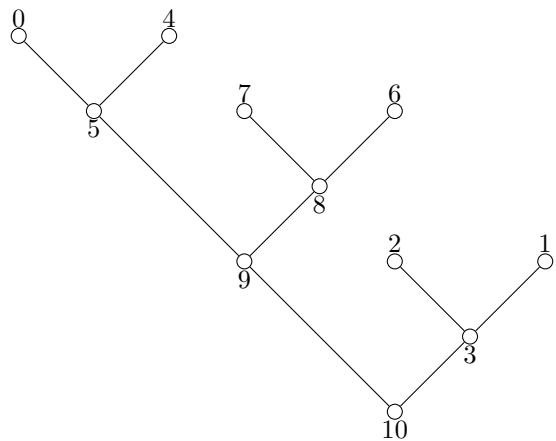
\begin{tikzpicture}

\node[inner sep=2pt, circle] (0) at (0,5) [draw] {};
\node[inner sep=2pt, circle] (1) at (7,2) [draw] {};
\node[inner sep=2pt, circle] (2) at (5,2) [draw] {};
\node[inner sep=2pt, circle] (3) at (6,1) [draw] {};
\node[inner sep=2pt, circle] (4) at (2,5) [draw] {};
\node[inner sep=2pt, circle] (5) at (1,4) [draw] {};
\node[inner sep=2pt, circle] (6) at (5,4) [draw] {};
\node[inner sep=2pt, circle] (7) at (3,4) [draw] {};
\node[inner sep=2pt, circle] (8) at (4,3) [draw] {};
\node[inner sep=2pt, circle] (9) at (3,2) [draw] {};
\node[inner sep=2pt, circle] (10) at (5,0) [draw] {};

\draw[-]  (0)--(5) node[midway, below] {};
\draw[-]  (4)--(5) node[midway, below] {};
\draw[-]  (5)--(9) node[midway, below] {};
\draw[-]  (7)--(8) node[midway, below] {};
\draw[-]  (6)--(8) node[midway, below] {};
\draw[-]  (8)--(9) node[midway, below] {};
\draw[-]  (2)--(3) node[midway, below] {};
\draw[-]  (1)--(3) node[midway, below] {};
\draw[-]  (3)--(10) node[midway, below] {};
\draw[-]  (10)--(9) node[midway, below] {};

\node[anchor = south]  at (0) {{$0$}};
\node[anchor = south]  at (1) {{$1$}};
\node[anchor = south]  at (2) {{$2$}};
\node[anchor = north]  at (3) {{$3$}};
\node[anchor = south]  at (4) {{$4$}};
\node[anchor = north]  at (5) {{$5$}};
\node[anchor = south]  at (6) {{$6$}};
\node[anchor = south]  at (7) {{$7$}};
\node[anchor = north]  at (8) {{$8$}};
\node[anchor = north]  at (9) {{$9$}};
\node[anchor = north]  at (10) {{$10$}};
\end{tikzpicture}
$$

\end{ex}

\begin{definition}\label{def: merege equiv} Two discrete Morse functions $f,g\colon T \to \RR$ are \textbf{merge equivalent} if they induce the same unlabeled binary tree; that is, if there is there is a rooted graph isomorphism $\phi\colon M_f \to M_g$ such that node $v$ is a left (right) child of node $u$ if and only if node $\phi(v)$ is a left (right) child of node $\phi(u)$.
\end{definition}

We will compare this notion of equivalence of discrete Morse functions with other notions of equivalence of discrete Morse functions in Section \ref{Comparison with other notions of equivalence}.

In addition, our main goal of Section \ref{Merge tree of a star graph} will be to count the number of merge equivalent discrete Morse functions on a star graph.

\begin{rem} Definition \ref{def: merege equiv} defines two merge trees to be equivalent if they share the same tree structure, ignoring the lifespan of a component and the order in which components were born and died. One could define a notion of equivalence that takes this order into account, thereby defining a chiral merge tree.  This was defined and studied in the smooth setting by Curry \cite{Curry-2017}.  We pose this adaption in the discrete setting as an open problem in Section \ref{Conclusion and future work}.

\end{rem}

\subsection{Relation to Matching number}

Recall that a \textbf{matching} in a graph is a set of edges such that no two edges share a common vertex. A matching is said to be \textbf{maximum} if it is a matching that contains the largest possible number of edges.  The \textbf{matching number} of $G$, denoted $\nu(G)$, is the size of a maximum matching.  We give a relationship between the matching number of a tree $T$ and the induced merge tree of any discrete Morse function on $T$ in Proposition \ref{prop: impasse matching}.  First a definition.

\begin{defn}
Let $M$ be a merge tree. An internal node of $M$ that is adjacent to exactly two leaves is called an \textbf{impasse}. The value $i(M)$ is the number of impasses of $M$.
\end{defn}

\begin{lemma}\label{lem: one impasse}
Every merge tree with more than one vertex has at least one impasse.  That is, $i(M) \ge 1$ for every merge tree $M$.
\end{lemma}
\begin{proof}
Suppose we have a merge tree $M$ without an impasse.  Therefore, all internal nodes of $M$ must have at least one internal node as a child.  Consequently, each of those internal nodes must now have an internal node as a child.  This continues on indefinitely, contradicting the fact that $M$ is finite.
\end{proof}

\begin{prop}\label{prop: impasse matching} Let $f\colon T \to \RR$ be discrete Morse function, $M$ the induced merge tree of $f$.  Then the set of edges of $T$ corresponding to the set of impasses of $M$ form a matching of $T$.  In particular, $i(M)\leq \nu(T)$.
\end{prop}

\begin{proof} Let $x,y$ be two impasses of $M$.  In particular, $x,y$ are not leaves and correspond to edges $e_x,e_y$, respectively, in $T$. We must show that $e_x$ and $e_y$ do not share a vertex. Suppose by contradiction that $e_x=uv$ and $e_y=uw$. Then the two leaves of $x$ must be nodes corresponding to $u$ and $v$, say $n_u$ and $n_v$.  Likewise for $y$. But then $n_u$ is adjacent to both $x$ and $y$, contradicting the fact that $n_u$ is a leaf.  Thus $e_x$ and $e_y$ do not share a vertex in common.  It follows that the corresponding set of edges forms a matching, hence $i(M)\leq \nu(T)$.
\end{proof}

\begin{example}\label{example impasse matching} The inequality in Proposition \ref{prop: impasse matching} can be strict.  Indeed, consider the following tree $T$ with discrete Morse function $f$

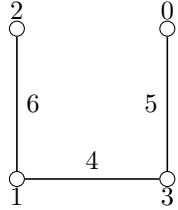
$$
\begin{tikzpicture}[scale=1]

\node[inner sep=2pt, circle] (0) at (2,2) [draw] {};
\node[inner sep=2pt, circle] (1) at (0,0) [draw] {};
\node[inner sep=2pt, circle] (2) at (0,2) [draw] {};
\node[inner sep=2pt, circle] (3) at (2,0) [draw] {};

\path[style=semithick] (0) edge node[anchor=east]{{$5$}}(3);
\path[style=semithick] (1) edge node[anchor=south]{{$4$}}(3);
\path[style=semithick] (2) edge node[anchor=west]{{$6$}}(1);

\node[anchor = south]  at (0) {{$0$}};
\node[anchor = north]  at (1) {{$1$}};
\node[anchor = south]  at (2) {{$2$}};
\node[anchor = north]  at (3) {{$3$}};

\end{tikzpicture}
$$

Then $f$ induces the merge tree $M$ given by

$$
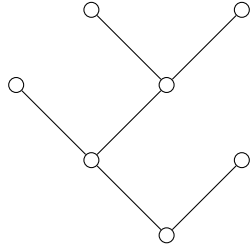
\begin{tikzpicture}

\node[inner sep=2pt, circle] (6) at (0,0) [draw] {};
\node[inner sep=2pt, circle] (5) at (-1,1) [draw] {};
\node[inner sep=2pt, circle] (2) at (1,1) [draw] {};
\node[inner sep=2pt, circle] (0) at (-2,2) [draw] {};
\node[inner sep=2pt, circle] (4) at (0,2) [draw] {};
\node[inner sep=2pt, circle] (1) at (-1,3) [draw] {};
\node[inner sep=2pt, circle] (3) at (1,3) [draw] {};

\draw[-]  (6)--(5) node[midway, below] {};
\draw[-]  (6)--(2) node[midway, below] {};
\draw[-]  (0)--(5) node[midway, below] {};
\draw[-]  (4)--(5) node[midway, below] {};
\draw[-]  (4)--(1) node[midway, below] {};
\draw[-]  (4)--(3) node[midway, below] {};
\end{tikzpicture}
$$

But clearly $i(M)=1<2=\nu(T)$.

\end{example}

\section{Comparison with other notions of equivalence}\label{Comparison with other notions of equivalence}

There are several other notions of equivalence of discrete Morse functions in the literature.  In this section, we compare merge equivalence with these other notions.

\subsection{Forman equivalence}

\begin{definition}\label{defn gvf} Let $f$ be a discrete Morse function on $G$.  The  \textbf{induced gradient vector field} $V_f$ is defined by
$$V_f:=\{(\sigma^{(p)}, \tau^{(p+1)}) : \sigma< \tau, f(\sigma)\geq f(\tau)\}.
$$
\end{definition}
Recall that two discrete Morse functions $f,g\colon G \to \RR$ are \textbf{Forman equivalent} if and only if $V_f=V_g$ \cite{Ayala09}. It is easy to see that neither Forman equivalence nor merge equivalence implies the other.

\begin{ex}\label{ex forman}
Suppose we have the following discrete Morse functions:
$$
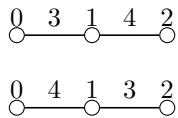
\begin{tikzpicture}
\node[inner sep=2pt, circle] (0) at (-1,0) [draw] {};
\node[inner sep=2pt, circle] (1) at (0,0) [draw] {};
\node[inner sep=2pt, circle] (2) at (1,0) [draw] {};

\path[style=semithick] (0) edge node[anchor=south]{{$3$}}(1);
\path[style=semithick] (1) edge node[anchor=south]{{$4$}}(2);

\node[anchor = south]  at (0) {{$0$}};
\node[anchor = south]  at (1) {{$1$}};
\node[anchor = south]  at (2) {{$2$}};
\end{tikzpicture}$$
$$
\begin{tikzpicture}
\node[inner sep=2pt, circle] (0) at (-1,0) [draw] {};
\node[inner sep=2pt, circle] (1) at (0,0) [draw] {};
\node[inner sep=2pt, circle] (2) at (1,0) [draw] {};

\path[style=semithick] (0) edge node[anchor=south]{{$4$}}(1);
\path[style=semithick] (1) edge node[anchor=south]{{$3$}}(2);

\node[anchor = south]  at (0) {{$0$}};
\node[anchor = south]  at (1) {{$1$}};
\node[anchor = south]  at (2) {{$2$}};
\end{tikzpicture}
$$
All simplices for both functions are critical, and hence the gradient vector field induced by both these functions has no arrows.  Thus these functions are Forman equivalent. However, the merge trees are given by
$$
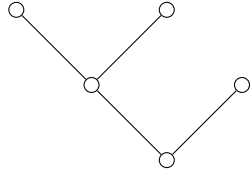
\begin{tikzpicture}

\node[inner sep=2pt, circle] (0) at (0,2) [draw] {};
\node[inner sep=2pt, circle] (1) at (2,2) [draw] {};
\node[inner sep=2pt, circle] (2) at (3,1) [draw] {};
\node[inner sep=2pt, circle] (3) at (1,1) [draw] {};
\node[inner sep=2pt, circle] (4) at (2,0) [draw] {};

\draw[-]  (0)--(3) node[midway, below] {};
\draw[-]  (3)--(1) node[midway, below] {};
\draw[-]  (3)--(4) node[midway, below] {};
\draw[-]  (2)--(4) node[midway, below] {};

\end{tikzpicture}
$$
and
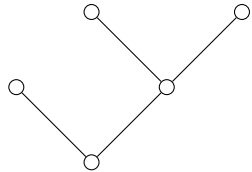
$$
\begin{tikzpicture}

\node[inner sep=2pt, circle] (0) at (0,1) [draw] {};
\node[inner sep=2pt, circle] (1) at (1,2) [draw] {};
\node[inner sep=2pt, circle] (2) at (3,2) [draw] {};
\node[inner sep=2pt, circle] (3) at (2,1) [draw] {};
\node[inner sep=2pt, circle] (4) at (1,0) [draw] {};

\draw[-]  (0)--(4) node[midway, below] {};
\draw[-]  (3)--(1) node[midway, below] {};
\draw[-]  (3)--(2) node[midway, below] {};
\draw[-]  (3)--(4) node[midway, below] {};,

\end{tikzpicture}
$$
respectively.  Thus they are not merge equivalent.

% On the other hand, consider the following merge tree
% $$
% \begin{tikzpicture}
% \node[inner sep=2pt, circle] (0) at (-1,0) [draw] {};
% \node[inner sep=2pt, circle] (1) at (0,-1) [draw] {};
% \node[inner sep=2pt, circle] (2) at (1,0) [draw] {};

% \draw[-]  (0)--(1) node[midway, below] {};
% \draw[-]  (1)--(2) node[midway, below] {};

% \end{tikzpicture}
% $$
% This merge tree is induced by both of the following discrete Morse functions:

% $$
% \begin{tikzpicture}
% \node[inner sep=2pt, circle] (0) at (-1,0) [draw] {};
% \node[inner sep=2pt, circle] (1) at (0,0) [draw] {};
% \node[inner sep=2pt, circle] (2) at (1,0) [draw] {};

% \path[style=semithick] (0) edge node[anchor=south]{{$2$}}(1);
% \path[style=semithick] (1) edge node[anchor=south]{{$3$}}(2);

% \node[anchor = south]  at (0) {{$0$}};
% \node[anchor = south]  at (1) {{$1$}};
% \node[anchor = south]  at (2) {{$3$}};

% \end{tikzpicture}
% $$
% $$
% \begin{tikzpicture}
% \node[inner sep=2pt, circle] (0) at (-1,0) [draw] {};
% \node[inner sep=2pt, circle] (1) at (0,0) [draw] {};

% \path[style=semithick] (0) edge node[anchor=south]{{$2$}}(1);

% \node[anchor = south]  at (0) {{$0$}};
% \node[anchor = south]  at (1) {{$1$}};
% \end{tikzpicture}
% $$
% However, these functions are not Forman equivalent, as one has a regular pair and the other does not.  Thus merge equivalence does not imply Forman equivalence.
\end{ex}

\subsection{Homological equivalence}

Given a graph with a discrete Morse function, one may study the Betti numbers of the level subcomplexes induced by the critical values.  This gives rise to a non-negative sequence of integers.  Such a sequence is a \textbf{homological sequence} and two discrete Morse functions are \textbf{homologically equivalent} if they induce the same homological sequence.  See \cite{A-F-F-V-09, A-F-V-11,Ayala10,AGOSR}.

% \begin{defn} Let $f$ be a discrete Morse function with $m$ critical values on an $n$-dimensional simplicial complex $K$.  The \index{homological sequence|textbf} \textbf{homological sequence of $f$} is given by the $n+1$ functions \index[not]{$B^f_i(k)$}
% $$B^f_0, B^f_1, \ldots, B^f_n\colon \{0, 1, \ldots, m-1\}\to
% \mathbb{N}\cup \{0\}$$
%  defined by $B^f_k(i):=b_k(K(c_i))$ for all $0\leq k\leq n$ and $0\leq i\leq m-1$.  We usually write $B_k(i)$ for $B_k^f(i)$ when the discrete Morse function $f$ is clear from the context.
% \end{defn}

% \begin{defn} Two discrete Morse functions $f,g\colon K \to \mathbb{R}$ with $m$ critical \index{homological equivalence|textbf} values are \textbf{homologically equivalent} if $B_k^f(i)=B_k^g(i)$ for all $0\leq k \leq m-1$ and $0\leq i$.
% \end{defn}

\begin{ex}\label{ex hom}

We show that homological equivalence and merge equivalence do not imply each other.  First, consider the first two discrete Morse functions from Example \ref{ex forman}.  It is easy to see that they both induce the homological sequence $1,2,3,2,1$, hence they are homologically equivalent.  However, their merge trees were shown in that same example to be different, hence they are not merge equivalent.

Now suppose we have the following functions

$$
\begin{tikzpicture}
\node[inner sep=2pt, circle] (0) at (-1,0) [draw] {};
\node[inner sep=2pt, circle] (1) at (0,0) [draw] {};
\node[inner sep=2pt, circle] (2) at (1,0) [draw] {};

\path[style=semithick] (0) edge node[anchor=south]{{$3$}}(1);
\path[style=semithick] (1) edge node[anchor=south]{{$4$}}(2);

\node[anchor = south]  at (0) {{$0$}};
\node[anchor = south]  at (1) {{$1$}};
\node[anchor = south]  at (2) {{$2$}};
\end{tikzpicture}
$$

$$
\begin{tikzpicture}
\node[inner sep=2pt, circle] (0) at (-1,0) [draw] {};
\node[inner sep=2pt, circle] (1) at (0,0) [draw] {};
\node[inner sep=2pt, circle] (2) at (1,0) [draw] {};

\path[style=semithick] (0) edge node[anchor=south]{{$2$}}(1);
\path[style=semithick] (1) edge node[anchor=south]{{$4$}}(2);

\node[anchor = south]  at (0) {{$0$}};
\node[anchor = south]  at (1) {{$1$}};
\node[anchor = south]  at (2) {{$3$}};
\end{tikzpicture}
$$

The homological sequence for these functions is given by $1,2,3,2,1$ and $1,2,1,2,1$, respectively so that they are not homologically equivalent. But they both induce the merge tree
$$
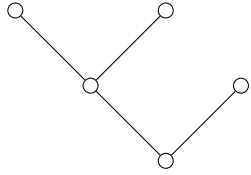
\begin{tikzpicture}

\node[inner sep=2pt, circle] (0) at (0,2) [draw] {};
\node[inner sep=2pt, circle] (1) at (2,2) [draw] {};
\node[inner sep=2pt, circle] (2) at (3,1) [draw] {};
\node[inner sep=2pt, circle] (3) at (1,1) [draw] {};
\node[inner sep=2pt, circle] (4) at (2,0) [draw] {};

\draw[-]  (0)--(3) node[midway, below] {};
\draw[-]  (3)--(1) node[midway, below] {};
\draw[-]  (3)--(4) node[midway, below] {};
\draw[-]  (2)--(4) node[midway, below] {};

\end{tikzpicture}
$$
so that they are merge equivalent.
\end{ex}

\subsection{Persistence equivalence}

Another notion of equivalence of discrete  Morse functions, closely related to merge equivalence, is persistence equivalence.  Two discrete Morse functions are \textbf{persistent equivalent} if they induce the same persistence diagram. See \cite{LiuSco-18} for more details.

\begin{ex}

Suppose we have the following discrete Morse functions:
$$
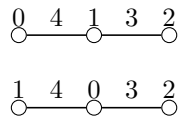
\begin{tikzpicture}
\node[inner sep=2pt, circle] (0) at (-1,0) [draw] {};
\node[inner sep=2pt, circle] (1) at (0,0) [draw] {};
\node[inner sep=2pt, circle] (2) at (1,0) [draw] {};

\path[style=semithick] (0) edge node[anchor=south]{{$4$}}(1);
\path[style=semithick] (1) edge node[anchor=south]{{$3$}}(2);

\node[anchor = south]  at (0) {{$0$}};
\node[anchor = south]  at (1) {{$1$}};
\node[anchor = south]  at (2) {{$2$}};
\end{tikzpicture}
$$
$$
\begin{tikzpicture}
\node[inner sep=2pt, circle] (0) at (0,0) [draw] {};
\node[inner sep=2pt, circle] (1) at (-1,0) [draw] {};
\node[inner sep=2pt, circle] (2) at (1,0) [draw] {};

\path[style=semithick] (0) edge node[anchor=south]{{$4$}}(1);
\path[style=semithick] (0) edge node[anchor=south]{{$3$}}(2);

\node[anchor = south]  at (0) {{$0$}};
\node[anchor = south]  at (1) {{$1$}};
\node[anchor = south]  at (2) {{$2$}};
\end{tikzpicture}
$$

These functions both create the same persistence diagram,
$$
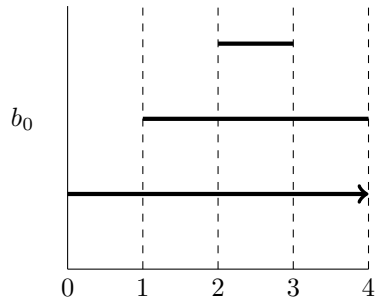
\begin{tikzpicture}[scale=.5][thick]
    \draw (1,1) -- (9,1);
    \draw (1,1) -- (1,8);
    \draw[thin, dashed] (3,1) -- (3,8);
    \draw[thin, dashed] (5,1) -- (5,8);
    \draw[thin, dashed] (7,1) -- (7,8);
    \draw[thin, dashed] (9,1) -- (9,8);

    \node at (1,.5) {$0$};
    \node at (3,.5) {$1$};
    \node at (5,.5) {$2$};
    \node at (7,.5) {$3$};
    \node at (9,.5) {$4$};

    \node[black] at (-.2,5) {$b_0$};

    \draw[black, ultra thick, ->] (1,3) -- (9,3);
    \draw[black, ultra thick] (3,5) -- (9,5);
    \draw[black, ultra thick] (5,7) -- (7,7);

\end{tikzpicture}
$$
but different merge trees.

The same functions in Example \ref{ex hom} which show that merge equivalence does not imply homological equivalence also shows that merge equivalence does not imply persistence equivalence.
\end{ex}

\section{Merge tree of a star graph}\label{Merge tree of a star graph}

\begin{defn} Let $n\geq 2$ be an integer. The \textbf{star graph on $n$ vertices} is defined by $S_n= K_{1,n-1}$ (\cite[p. 17]{H-69}).  We call the unique vertex $c\in S_n$ of degree $n-1$ the \textbf{center of $S_n$} or \textbf{center vertex}.
\end{defn}

In this section, we consider discrete Morse functions on a star graph with every simplex critical.  We will call such a function a \textbf{critical} discrete Morse function.  We first define the kind of merge tree that, it turns out, can be induced by a critical discrete Morse function on a star graph.

\begin{defn} A merge tree $M$ is called \textbf{thin} if $i(M)=1$, i.e, $M$ has a unique impasse.
\end{defn}

Observe that a thin merge tree can be characterized by the fact that it has a unique path from the root node to the unique impasse with the property that every edge not on the path and incident with the path is part of a leaf.  We are thus able to determine a unique thin merge tree through a sequence of Ls and Rs where the L or R is specifying the next child to travel to, i.e., the direction that this path takes.  We make this notion precise in the following definition.

\begin{definition}\label{def: unique path notation} Let $M$ be a thin merge tree with $n$ leaves, and let $P_M$ denote the unique path $r=m_1, m_2, \ldots, m_{n-1}$ from the root node $r$ to the unique impasse $m_{n-1}$.  Define a function $d_M=d\colon\{m_1, \ldots, m_{n-2}\}\to \{L,R\}$ by $d(m_i)=L$ ( or $R)$ if $m_{i+1}$ is the left (or right) child of $m_i$, $1\leq i\leq n-2$. The \textbf{LR sequence} of $M$, denoted $D_M=D$, is the sequence $D\colon \{0,1,2,3, \ldots, n-2\}\to \{L,R\}$ by $D(0)=L$ and $D_M(i):=d(m_i)$ for all $1\leq i \leq n-2$.
\end{definition}

The choice that $D(0)=L$ is arbitrary, but chosen to be consistent with the construction in Theorem \ref{thm: construct mereg tree}. We will also need this in the proof of Proposition \ref{prop: dmf from thin} to determine when the sequence $D$ switches direction.

\begin{ex} We illustrate the terms and notation in Definition \ref{def: unique path notation}.  Consider the merge tree $M$

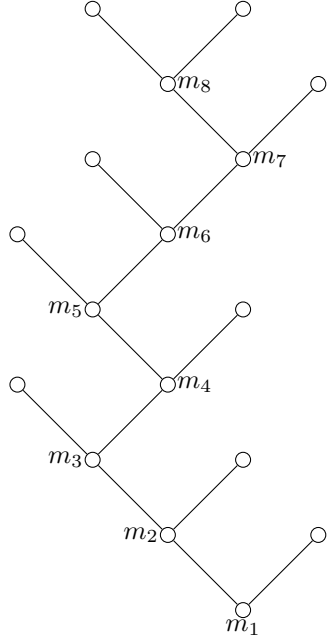
$$
\begin{tikzpicture}

\node[inner sep=2pt, circle] (0) at (1,9) [draw] {};
\node[inner sep=2pt, circle] (1) at (3,9) [draw] {};
\node[inner sep=2pt, circle] (2) at (2,8) [draw] {};
\node[inner sep=2pt, circle] (3) at (4,8) [draw] {};
\node[inner sep=2pt, circle] (4) at (3,7) [draw] {};
\node[inner sep=2pt, circle] (5) at (1,7) [draw] {};
\node[inner sep=2pt, circle] (6) at (2,6) [draw] {};
\node[inner sep=2pt, circle] (7) at (0,6) [draw] {};
\node[inner sep=2pt, circle] (8) at (1,5) [draw] {};
\node[inner sep=2pt, circle] (9) at (3,5) [draw] {};
\node[inner sep=2pt, circle] (10) at (2,4) [draw] {};
\node[inner sep=2pt, circle] (11) at (0,4) [draw] {};
\node[inner sep=2pt, circle] (12) at (1,3) [draw] {};
\node[inner sep=2pt, circle] (13) at (3,3) [draw] {};
\node[inner sep=2pt, circle] (14) at (2,2) [draw] {};
\node[inner sep=2pt, circle] (15) at (4,2) [draw] {};
\node[inner sep=2pt, circle] (16) at (3,1) [draw] {};

\draw[-]  (0)--(2) node[midway, below] {};
\draw[-]  (1)--(2) node[midway, below] {};
\draw[-]  (2)--(4) node[midway, below] {};
\draw[-]  (3)--(4) node[midway, below] {};
\draw[-]  (5)--(6) node[midway, below] {};
\draw[-]  (6)--(4) node[midway, below] {};
\draw[-]  (7)--(8) node[midway, below] {};
\draw[-]  (6)--(8) node[midway, below] {};
\draw[-]  (8)--(10) node[midway, below] {};
\draw[-]  (10)--(9) node[midway, below] {};
\draw[-]  (11)--(12) node[midway, below] {};
\draw[-]  (12)--(10) node[midway, below] {};
\draw[-]  (12)--(14) node[midway, below] {};
\draw[-]  (13)--(14) node[midway, below] {};
\draw[-]  (14)--(16) node[midway, below] {};
\draw[-]  (15)--(16) node[midway, below] {};

\node[anchor = north]  at (16) {{$m_1$}};
\node[anchor = east]  at (14) {{$m_2$}};
\node[anchor = east]  at (12) {{$m_3$}};
\node[anchor = west]  at (10) {{$m_4$}};
\node[anchor = east]  at (8) {{$m_5$}};
\node[anchor = west]  at (6) {{$m_6$}};
\node[anchor = west]  at (4) {{$m_7$}};
\node[anchor = west]  at (2) {{$m_8$}};

\end{tikzpicture}
$$

Then $d(m_1)=d(m_2)=d(m_4)=d(m_7)=L$ while $d(m_3)=d(m_5)=d(m_6)=R$.
We then have that the LR sequence of $M$ is (L)LLRLRRL. Note that the $0$ term of the sequence is always L, which we will put in parentheses when it is included.
\end{ex}

By choosing a direction to travel, we easily count all thin merge trees with $n$ internal nodes.

\begin{prop}\label{prop: counting thin trees}
There are exactly $2^{n-1}$ thin merge trees with $n$ internal nodes.
\end{prop}

\begin{proof}
We will count the number of ways to construct a thin merge tree with $n$ internal nodes.   Beginning with the root node, we will choose which of its neighbors, left or right, is the next internal node, forcing the other node to be a leaf, since by definition, a thin tree has only one internal node incident two leaves with all other internal nodes being incident to one leaf and one internal node.  At each internal node, we choose the next internal node to be on the left or the right.  This choice is made for every internal node except for the very last internal node, which ends with the node adjacent to two leaves.  Since this choice is made for all internal nodes except the impasse, there are exactly $2^{n-1}$ thin merge trees with $n$ internal nodes.
\end{proof}

The goal of the remainder of the section is to show that thin merge trees with $n+1$ leaves are in bijective correspondence with the merge equivalent discrete Morse functions on $S_n$.  As promised, the next proposition tells us that the merge tree induced by a discrete Morse function on a star graph is always thin.

\begin{prop}\label{prop: star impass}
If $T=S_n$ is a star graph and $f\colon S_n\to \RR$ any discrete Morse function, then $i(M_f)= 1$.
\end{prop}

\begin{proof} By Proposition \ref{prop: impasse matching}, $i(M_f)\leq \nu(S_n)=1$.  To see that we have equality, apply Lemma \ref{lem: one impasse}.
\end{proof}

Finally, given a thin merge tree with $n$ internal nodes, we need to construct a critical discrete Morse function on a star graph whose induced merge tree is the given thin merge tree.

\begin{prop}\label{prop: dmf from thin} Suppose $M$ is a thin merge tree.  Then there is a star graph $S_n$ and a  discrete Morse function $f\colon S_n\to M$ such that $M_f=M.$
\end{prop}

\begin{proof}
Let $M$ be a thin merge tree with $n+1$ leaves and choose $T = S_n$.  We construct $f\colon S_n \to \RR$ such that $M_f = M$.  Consider the path $r=m_1, m_2, \ldots m_{n-1}$ from the root node $r$ to the unique impasse $m_{n-1}$ of $M$.  Now let
\begin{equation} \label{eq:1}
m_{i_1}, m_{i_2}, \ldots, m_{i_k}
\end{equation}
\noindent be the (possibly empty) nodes on the path such that $d_M(m_{i_j})\neq d_M(m_{i_j-1})$, that is, the nodes in the path that switch direction.  To keep track of the labeling and correspondence between $M$ and $S_n$, we will also label nodes of $M$. To that end, label the unique leaf of $m_{i_j}$ with label $j$ for $1\leq j \leq k-1$ and either of the leafs of the unique impasse $m_{n-1}$ with label $k$. Now choose $k-1$ corresponding non-center vertices of $S_n$ and label them $1,2, \ldots, k-1$. Choose the center vertex of $S_n$ to label $k$.  Traversing the path from $m_{n-1}$ to $m_1$, label each unlabeled leaf $k+1, k+2, \ldots, k+\ell=n+1$ in order and label the remaining vertices of $S_n$ by $k+1, \ldots, k+\ell$. Finally, we know that each internal node corresponds to the edge in $S_n$.  Traversing the path from $m_{n-1}$ to $m_1$ again, label in order each internal node $k+\ell+1, k+\ell+2, \ldots, k+\ell+d=2n+1$. The node $m_k$ is labeled $k+\ell+1$ and has children labeled $k$ and $k+1$.  Hence label the edge in $S_n$ connecting vertices $k$ and $k+1$ with label $k+\ell+1$. Continuing along the path $m_{n-1}$ to $m_1$, suppose we are at node $n_{k-i}$ with label $k+\ell+i$.  Then $m_{k-i}$ is adjacent to a leaf.  This leaf corresponds to a vertex in $S_n$ which is incident to a unique edge.  Label this edge $k+\ell+i$.  This completes the labeling of $S_n$.

It remains to show that the labeling described above is a critical discrete Morse function whose induced merge tree is the given thin merge tree $M$.  By construction, the vertices of $S_n$  are labeled first, and since each subsequent label is strictly greater than the previous label, each vertex has a value strictly less than each edge, and hence $f$ is a critical discrete Morse function. By Proposition \ref{prop: star impass}, the merge tree induced by a discrete Morse function on a star graph is a thin merge tree.  Hence it suffices to show that the LR sequence of $M$ agrees with the LR sequence of $M_f$, i.e., $D_M=D_{M_f}$.

We proceed inductively on the LR sequence, showing that $D_M(i+1)=D_M(i)$ if and only if $D_{M_f}(i+1)=D_{M_f}(i)$ for $0\leq i\leq n-2$, i.e., the LR sequence for $M_f$ changes directions precisely when the LR sequence for $M$ does. By both constructions, the root node for both $M$ and $M_f$ is given direction L by definition, i.e, $D_M(0)=D_{M_f}(0)=$L. Inductively, suppose the LR sequence for $M$ agrees with the LR sequence for $M_f$ up to $i$, $0\leq i\leq n-2$, and consider the node $m_{i+1}$ on the path $P_M$. Let $e$ be the edge of $S_n$ corresponding to node $m_{i+1}$.  Then $T_{f(e)-\ep}$ is a forest where the components of the endpoints of $e$ are a star graph $S_{i+1}$ and an isolated vertex $v$. Now this $v$ induces a leaf node in the construction of $M_f$.  Hence the path $P(M_f)$ will follow the direction of the node whose label is the maximum value of the star graph $S_{i+1}$. By the construction of Theorem \ref{thm: construct mereg tree}, the forest component with the smallest value will share the direction of the previous node. We proceed by cases.\\

\noindent \textit{Case 1:  $D_M(i+1)=D_M(i)$}

Suppose $D_M(i+1)=D_M(i)$. Then we need to show that there is a vertex of $S_{i+1}$ with label less than that of $v$. Since $d_M(m_i)=d_M(m_{i+1})$, it follows by construction that $m_{i+1}$ is not any of the $m_{i_j}$ from Equation (\ref{eq:1}). Thus, by the labeling of $M$, the leaf node of $m_{i+1}$ is given a value $t$ which is strictly greater than the value of either leaf of the impasse $m_k$. But by construction, $f(v)=t$ and the values assigned to the leaf nodes of $m_k$ are assigned to corresponding vertices of $S_{i+1}$. We conclude that $D_{M_f}(i+1)=D_{M_f}(i)$. \\

\noindent \textit{Case 2:  $D_M(i+1)\neq D_M(i)$}

Now suppose that $D_M(i+1)\neq D_M(i)$.  We need to show that $f(v)<f(u)$ for all $u\in S_{i+1}$. Since $d_M(m_i)\neq d_M(m_{i+1})$, then $m_{i+1}=m_{i_j}$ from Equation (\ref{eq:1}) for some $j$.  Hence by construction the leaf node of $m_{i+1}$ is given a value strictly less than the other $m_{i_k}$ for all $k>j$.  Since none of the leaf nodes of $m_{i_p}$, $p<k$, have corresponding vertices in $S_{i+1}$, the value $f(v)<f(u)$ for all $u\in S_{i+1}$. Thus $D_{M_f}(i+1)\neq D_{M_f}(i)$.  This completes the proof.

% Recall by the construction in Theorem \ref{thm: construct mereg tree} that the root vertex of $M_f$ is given value $\max\{\mathrm{Im}(f)\}=f(e)=2n+1$ for some edge $e$ along with direction L. Writing $S_n=T$, we then see that $T_{f(e)-\ep}$ is a forest consisting of the star graph $S_{n-1}$ along with an isolated vertex. Clearly the node that is labeled corresponding to the isolated vertex will be a leaf of $M_f$, hence the path $P(M_f)$ will follow the direction of the node which is labeled the maximum value of $S_{n-1}$.  Again, by the construction of Theorem \ref{thm: construct mereg tree}, the forest component with the smallest value will share the direction (L) of the root node. Suppose $d_M(m_2)=d_M(m_1)$.  Then $m_2 \neq m_{i_1}$ so that the vertex of $S_n$ corresponding to the node $m_2$ has is labeled with a value greater than the vertex of $S_n$ corresponding to the node $m_{i_1}$.  This latter vertex is in the $S_{n-1}$ component so that the path $P_{M_f}$ shares the same direction as $m_1$.   In the case where $d_M(m_2)\neq d_M(m_1)$, then $m_2 = m_{i_1}$, so that the isolated vertex of the forest $T_{f(e)-\ep}$ has the minimum value of $T_{f(e)-\ep}$.  Thus the node in $M_f$ corresponding to this vertex shares the direction with the root node, and hence the path $M_f$ changes direction. In either case, the first direction of $P_{M_f}$ agrees with the first direction of $P_M.$ The induction proceeds exactly as the base case and is hence omitted.
\end{proof}

\begin{cor}
For any star graph $S_n$, there are exactly $2^{n-1}$ possible merge trees induced by a critical discrete Morse function on $S_n$.
\end{cor}

\begin{example}
We give an example illustrating the construction of Proposition \ref{prop: dmf from thin}. We will find a star graph and discrete Morse function on that graph that induces the merge tree
$$
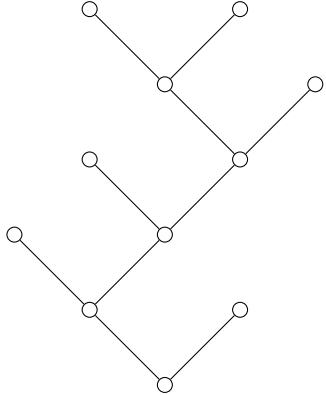
\begin{tikzpicture}

\node[inner sep=2pt, circle] (0) at (0,2) [draw] {};
\node[inner sep=2pt, circle] (1) at (4,4) [draw] {};
\node[inner sep=2pt, circle] (2) at (1,5) [draw] {};
\node[inner sep=2pt, circle] (3) at (3,5) [draw] {};
\node[inner sep=2pt, circle] (4) at (1,3) [draw] {};
\node[inner sep=2pt, circle] (5) at (3,1) [draw] {};
\node[inner sep=2pt, circle] (6) at (2,4) [draw] {};
\node[inner sep=2pt, circle] (7) at (3,3) [draw] {};
\node[inner sep=2pt, circle] (8) at (2,2) [draw] {};
\node[inner sep=2pt, circle] (9) at (1,1) [draw] {};
\node[inner sep=2pt, circle] (10) at (2,0) [draw] {};

\draw[-]  (0)--(9) node[midway, below] {};
\draw[-]  (9)--(8) node[midway, below] {};
\draw[-]  (9)--(10) node[midway, below] {};
\draw[-]  (5)--(10) node[midway, below] {};
\draw[-]  (8)--(4) node[midway, below] {};
\draw[-]  (7)--(8) node[midway, below] {};
\draw[-]  (7)--(1) node[midway, below] {};
\draw[-]  (6)--(7) node[midway, below] {};
\draw[-]  (6)--(2) node[midway, below] {};
\draw[-]  (6)--(3) node[midway, below] {};

\end{tikzpicture}
$$

We first observe that because there are $6$ leaves, we choose $T:=S_5$. Beginning at the root vertex of $M$, we arrive at the impasse through the sequence of moves (L)LRRL.  At each switch from L to R or R to L, we label the corresponding leaf with with the next integer value yielding

$$
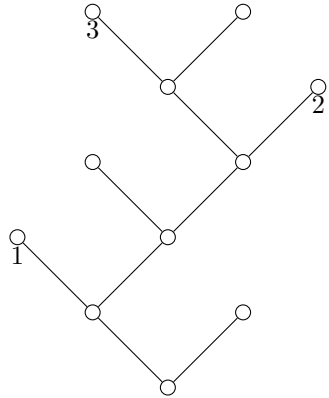
\begin{tikzpicture}

\node[inner sep=2pt, circle] (0) at (0,2) [draw] {};
\node[inner sep=2pt, circle] (1) at (4,4) [draw] {};
\node[inner sep=2pt, circle] (2) at (1,5) [draw] {};
\node[inner sep=2pt, circle] (3) at (3,5) [draw] {};
\node[inner sep=2pt, circle] (4) at (1,3) [draw] {};
\node[inner sep=2pt, circle] (5) at (3,1) [draw] {};
\node[inner sep=2pt, circle] (6) at (2,4) [draw] {};
\node[inner sep=2pt, circle] (7) at (3,3) [draw] {};
\node[inner sep=2pt, circle] (8) at (2,2) [draw] {};
\node[inner sep=2pt, circle] (9) at (1,1) [draw] {};
\node[inner sep=2pt, circle] (10) at (2,0) [draw] {};

\draw[-]  (0)--(9) node[midway, below] {};
\draw[-]  (9)--(8) node[midway, below] {};
\draw[-]  (9)--(10) node[midway, below] {};
\draw[-]  (5)--(10) node[midway, below] {};
\draw[-]  (8)--(4) node[midway, below] {};
\draw[-]  (7)--(8) node[midway, below] {};
\draw[-]  (7)--(1) node[midway, below] {};
\draw[-]  (6)--(7) node[midway, below] {};
\draw[-]  (6)--(2) node[midway, below] {};
\draw[-]  (6)--(3) node[midway, below] {};

\node[anchor = north]  at (0) {{$1$}};
\node[anchor = north]  at (1) {{$2$}};
\node[anchor = north]  at (2) {{$3$}};

\end{tikzpicture}
$$

This corresponds to vertices in $S_5$, with the last value given to the center vertex

$$
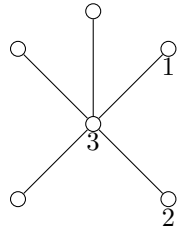
\begin{tikzpicture}

\node[inner sep=2pt, circle] (0) at (1,1) [draw] {};
\node[inner sep=2pt, circle] (1) at (1,-1) [draw] {};
\node[inner sep=2pt, circle] (2) at (0,0) [draw] {};
\node[inner sep=2pt, circle] (3) at (0,1.5) [draw] {};
\node[inner sep=2pt, circle] (4) at (-1,1) [draw] {};
\node[inner sep=2pt, circle] (5) at (-1,-1) [draw] {};

\draw[-]  (2)--(0) node[midway, below] {};
\draw[-]  (2)--(1) node[midway, below] {};
\draw[-]  (2)--(3) node[midway, below] {};
\draw[-]  (5)--(2) node[midway, below] {};
\draw[-]  (4)--(2) node[midway, below] {};

\node[anchor = north]  at (0) {{$1$}};
\node[anchor = north]  at (1) {{$2$}};
\node[anchor = north]  at (2) {{$3$}};

\end{tikzpicture}
$$

Now traverse the path in $M$ from the impasse to the root, labeling the leaves $4,5 \ldots$.
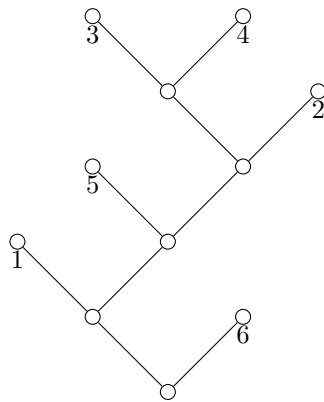
$$
\begin{tikzpicture}

\node[inner sep=2pt, circle] (0) at (0,2) [draw] {};
\node[inner sep=2pt, circle] (1) at (4,4) [draw] {};
\node[inner sep=2pt, circle] (2) at (1,5) [draw] {};
\node[inner sep=2pt, circle] (3) at (3,5) [draw] {};
\node[inner sep=2pt, circle] (4) at (1,3) [draw] {};
\node[inner sep=2pt, circle] (5) at (3,1) [draw] {};
\node[inner sep=2pt, circle] (6) at (2,4) [draw] {};
\node[inner sep=2pt, circle] (7) at (3,3) [draw] {};
\node[inner sep=2pt, circle] (8) at (2,2) [draw] {};
\node[inner sep=2pt, circle] (9) at (1,1) [draw] {};
\node[inner sep=2pt, circle] (10) at (2,0) [draw] {};

\draw[-]  (0)--(9) node[midway, below] {};
\draw[-]  (9)--(8) node[midway, below] {};
\draw[-]  (9)--(10) node[midway, below] {};
\draw[-]  (5)--(10) node[midway, below] {};
\draw[-]  (8)--(4) node[midway, below] {};
\draw[-]  (7)--(8) node[midway, below] {};
\draw[-]  (7)--(1) node[midway, below] {};
\draw[-]  (6)--(7) node[midway, below] {};
\draw[-]  (6)--(2) node[midway, below] {};
\draw[-]  (6)--(3) node[midway, below] {};

\node[anchor = north]  at (0) {{$1$}};
\node[anchor = north]  at (1) {{$2$}};
\node[anchor = north]  at (2) {{$3$}};
\node[anchor = north]  at (3) {{$4$}};
\node[anchor = north]  at (4) {{$5$}};
\node[anchor = north]  at (5) {{$6$}};

\end{tikzpicture}
$$
This corresponds to the same labels on the remaining leaves of $S_5$:

$$
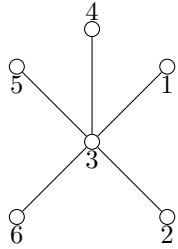
\begin{tikzpicture}

\node[inner sep=2pt, circle] (0) at (1,1) [draw] {};
\node[inner sep=2pt, circle] (1) at (1,-1) [draw] {};
\node[inner sep=2pt, circle] (2) at (0,0) [draw] {};
\node[inner sep=2pt, circle] (3) at (0,1.5) [draw] {};
\node[inner sep=2pt, circle] (4) at (-1,1) [draw] {};
\node[inner sep=2pt, circle] (5) at (-1,-1) [draw] {};

\draw[-]  (2)--(0) node[midway, below] {};
\draw[-]  (2)--(1) node[midway, below] {};
\draw[-]  (2)--(3) node[midway, below] {};
\draw[-]  (5)--(2) node[midway, below] {};
\draw[-]  (4)--(2) node[midway, below] {};

\node[anchor = north]  at (0) {{$1$}};
\node[anchor = north]  at (1) {{$2$}};
\node[anchor = north]  at (2) {{$3$}};
\node[anchor = south]  at (3) {{$4$}};
\node[anchor = north]  at (4) {{$5$}};
\node[anchor = north]  at (5) {{$6$}};

\end{tikzpicture}
$$

Traverse the same path in $M$, labeling the internal nodes $7, 8\ldots$
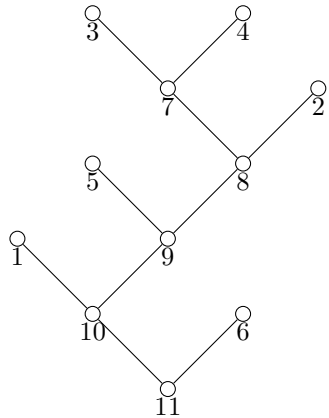
$$
\begin{tikzpicture}

\node[inner sep=2pt, circle] (0) at (0,2) [draw] {};
\node[inner sep=2pt, circle] (1) at (4,4) [draw] {};
\node[inner sep=2pt, circle] (2) at (1,5) [draw] {};
\node[inner sep=2pt, circle] (3) at (3,5) [draw] {};
\node[inner sep=2pt, circle] (4) at (1,3) [draw] {};
\node[inner sep=2pt, circle] (5) at (3,1) [draw] {};
\node[inner sep=2pt, circle] (6) at (2,4) [draw] {};
\node[inner sep=2pt, circle] (7) at (3,3) [draw] {};
\node[inner sep=2pt, circle] (8) at (2,2) [draw] {};
\node[inner sep=2pt, circle] (9) at (1,1) [draw] {};
\node[inner sep=2pt, circle] (10) at (2,0) [draw] {};

\draw[-]  (0)--(9) node[midway, below] {};
\draw[-]  (9)--(8) node[midway, below] {};
\draw[-]  (9)--(10) node[midway, below] {};
\draw[-]  (5)--(10) node[midway, below] {};
\draw[-]  (8)--(4) node[midway, below] {};
\draw[-]  (7)--(8) node[midway, below] {};
\draw[-]  (7)--(1) node[midway, below] {};
\draw[-]  (6)--(7) node[midway, below] {};
\draw[-]  (6)--(2) node[midway, below] {};
\draw[-]  (6)--(3) node[midway, below] {};

\node[anchor = north]  at (0) {{$1$}};
\node[anchor = north]  at (1) {{$2$}};
\node[anchor = north]  at (2) {{$3$}};
\node[anchor = north]  at (3) {{$4$}};
\node[anchor = north]  at (4) {{$5$}};
\node[anchor = north]  at (5) {{$6$}};
\node[anchor = north]  at (6) {{$7$}};
\node[anchor = north]  at (7) {{$8$}};
\node[anchor = north]  at (8) {{$9$}};
\node[anchor = north]  at (9) {{$10$}};
\node[anchor = north]  at (10) {{$11$}};

\end{tikzpicture}
$$

Finally, each edge of $S_5$ is labeled with the same label as the internal nodes of $M$.  If edge $e\in S_5$ is incident with non-center vertex labeled $a$, then the leaf in $M$ labeled $a$ is incident with an interior node labeled $b$.  Thus define $f(e)=b$ so that the discrete Morse function is
$$
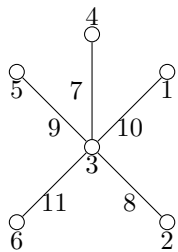
\begin{tikzpicture}

\node[inner sep=2pt, circle] (0) at (1,1) [draw] {};
\node[inner sep=2pt, circle] (1) at (1,-1) [draw] {};
\node[inner sep=2pt, circle] (2) at (0,0) [draw] {};
\node[inner sep=2pt, circle] (3) at (0,1.5) [draw] {};
\node[inner sep=2pt, circle] (4) at (-1,1) [draw] {};
\node[inner sep=2pt, circle] (5) at (-1,-1) [draw] {};

\draw[-]  (2)--(0) node[midway, below] {$10$};
\draw[-]  (2)--(1) node[midway, below] {$8$};
\draw[-]  (2)--(3) node[midway, left] {$7$};
\draw[-]  (5)--(2) node[midway, below] {$11$};
\draw[-]  (4)--(2) node[midway, below] {$9$};

\node[anchor = north]  at (0) {{$1$}};
\node[anchor = north]  at (1) {{$2$}};
\node[anchor = north]  at (2) {{$3$}};
\node[anchor = south]  at (3) {{$4$}};
\node[anchor = north]  at (4) {{$5$}};
\node[anchor = north]  at (5) {{$6$}};

\end{tikzpicture}
$$

\end{example}

\section{Future directions and open questions}\label{Conclusion and future work}

In this final section, we share some ideas for future directions that one could pursue.

We were able to compute all merge trees induced by a discrete Morse function on a star graph.  What other classes of merge trees can be realized?  Can one characterize the set of all merge equivalent discrete Morse functions on a caterpillar graph, regular tree, binary tree, or other class of trees? Conversely, given a class of merge trees with some special property, can we find a characterization of the graphs whose set of all merge equivalent discrete Morse functions induces this class of merge trees?

Can any merge tree be realized by a discrete Morse function on some graph?  We conjecture that this can be done with the right discrete Morse function on a path. Conversely, we conjecture that if $T$ is a tree containing a vertex with degree greater than $2$, then there exists a merge tree that cannot be realized by any discrete Morse function on $T$.

We have seen that there is a one to one correspondence between thin merge trees on $n+1$ leaves and merge equivalent discrete Morse functions on $S_n$.  Thin merge trees, however, do not characterize star graphs, as Example \ref{example impasse matching} shows that a non-star graph can induce a thin merge tree.  Is there a substantive inverse problem in this setting?  That is, given some class of merge trees induced by a discrete Morse function on a tree or a tree up to some notion of equivalence, can we use this class of merge trees to reconstruct the isomorphism type of the tree?

There are variations of equivalence classes of merge trees that can be considered.  For example, one could require that there be a total ordering on the vertices of the merge tree, a so-called chiral merge tree \cite{Curry-2017}.  Or one could think of of each edge as having a length by putting some weight on the edges.  One could also relax the condition that merge trees require a distinction between the left and right children.

\bibliographystyle{amsplain}
\bibliography{Merge}

\end{document}